\documentclass[12pt,a4paper]{article}%
\usepackage{amsmath, amsfonts, amsthm,color, latexsym, amssymb}
\usepackage{amscd}
\usepackage{amsmath}
\usepackage{amsfonts}
\usepackage{amssymb}
\usepackage{graphicx}%
\providecommand{\U}[1]{\protect\rule{.1in}{.1in}}
\newtheorem{theorem}{Theorem}[section]
\newtheorem{proposition}[theorem]{Proposition}
\newtheorem{corollary}[theorem]{Corollary}

\newtheorem{remark}[theorem]{Remark}

\begin{document}

\title{\textsc{Cotype and absolutely summing linear operators}}
\author{Geraldo Botelho, Daniel Pellegrino\thanks{Supported by CNPq Grant
308084/2006-3.}~ and Pilar Rueda\thanks{Supported by Ministerio de
Ciencia e Innovaci\'{o}n MTM2008-03211/MTM.\hfill\newline2000
Mathematics Subject Classification: 46B20, 47B10.}}
\date{}
\maketitle

\begin{abstract}
Cotype is used in this paper to prove new results concerning the
existence of approximable (hence compact) non-absolutely summing linear operators between Banach spaces. We derive consequences that
extend/generalize/complement some classical results of Bennett \cite{bennett}, Davis and Johnson \cite{davis}, and Lindenstrauss and Pe\l czy\'{n}ski \cite{Lindenstrauss}. We also point
out that some of our results are sharp.
\end{abstract}

\vspace*{-1.0em}

\section*{Introduction}
The central role played by the concept of cotype in the theory of
absolutely summing linear (and nonlinear) operators is well known
(the linear case is folklore, for the nonlinear case, see, e.g.,
\cite{B, studia}). Cotype has been traditionally used to show that
operators between certain spaces are always absolutely summing, that
is, cotype has been a valuable tool to show how abundant absolutely
summing operators are. In this paper we show how cotype can also be
used to prove the existence of non-absolutely summing operators. For
instance we prove in Corollary \ref{corol} that if $X$ and $Y$ are
infinite-dimensional Banach spaces and $\cot Y>\max\{2,r\}$, then
there exists an approximable (hence compact, weakly compact, completely continuous, strictly singular, strictly cosingular, etc) linear operator from $X$ to $Y$ that fails to
be $r$-summing. We also show that some of our results are somehow sharp. \\
\indent These results and their consequences follow the line of
classical results of Bennett \cite{bennett}, Davis and Johnson
\cite{davis}, Lindenstrauss and Pe\l czy\'{n}ski \cite{Lindenstrauss} and
complement/generalize
some of their information. More precisely:\\
$\bullet$ Non-coincidence results for absolutely summing operators
between $\ell_p$-spaces of \cite{bennett} are shown to hold true for
operators between much more general spaces (Proposition
\ref{propos});\\
$\bullet$ The existence of compact non-absolutely summing operators
on superreflexive spaces of \cite{davis} is extended to operators on
arbitrary infinite-dimensional spaces provided that the infimum of the cotypes of the
range spaces is not 2 (Corollary \ref{corol});\\
$\bullet$ In \cite{Lindenstrauss} it is proved that if every
operator from $X$ to $Y$ is absolutely summing, then $\cot X = 2$;
in Proposition \ref{corol1}(c) we improve this result by showing that with weaker conditions we
can conclude that $\cot X = \cot Y = 2$.\\
\indent If every continuous linear operator from $X$ to any Banach space $Y$ is $(q,1)$-summing, we can call on a deep result due to Maurey and Pisier \cite{MPi} (cf. \cite[Theorem 14.5]{DJT}) to conclude that $\cot X \leq q$. In Corollary \ref{gghh} we combine Maurey and Pisier's result with Theorem \ref{jun5} to show that the conclusion  $\cot X \leq q$ can be derived from the weaker condition that every {\it approximable} linear operator from $X$ to {\it some} infinite-dimensional Banach space $Y$ with no finite cotype is $(q,1)$-summing.

\section{Notation and background}

Throughout this paper $X$ and $Y$ will
stand for Banach spaces over $\mathbb{K}=\mathbb{R}$ or $\mathbb{C}$ and ${\cal L}(X,Y)$ means the space of bounded linear operators from $X$ to $Y$ endowed with the usual sup norm. The closed unit ball of $X$ is denoted by $B_X$ and the topological dual of $X$ by $X^*$. We follow the standard terminology of absolutely summing operators (see \cite{DJT}). Given $x_1, \ldots, x_n \in X$,  $$\|(x_j)_{j=1}^n\|_{w,p} = \sup_{\varphi \in B_{X^*}} \left( \sum_{j=1}^n |\varphi(x_j)|^p \right)^\frac1p.$$
A linear operator $u \in {\cal L}(X,Y)$ is absolutely $(q,p)$-summing if there is a constant $C \geq 0$ such that
$$\left( \sum_{j=1}^n \|u(x_j)\|^q \right)^\frac1q \leq C \|(x_j)_{j=1}^n\|_{w,p}$$
for every $n \in \mathbb{N}$ and $x_1, \ldots, x_n \in X$. The space of absolutely $(q,p)$-summing linear operators from $X$ to $Y$ is denoted by $\Pi_{q,p}(X,Y)$. The $(q,p)$-summing norm of $u$, defined as the infimum of the constants $C$ working in the inequality, is represented by $\pi_{q,p}(u)$. If $p = q$ we say simply absolutely $p$-summing operator and write $\Pi_{p}(X,Y)$ and $\pi_{p}(u)$ for the space of absolutely $p$-summing operators and the $p$-summing norm of $u$, respectively.\\
\indent The notation $\cot X$ denotes the infimum of the cotypes
assumed by $X$. The identity operator on $X$ is denoted by
$id_{X}$.\\
\indent By ${\cal A}(X,Y)$ we mean the subspace of ${\cal L}(X,Y)$ of approximable operators, that is, the closure of the finite rank operators in the usual operator sup-norm.

\section{Results}\label{vectorvalued}

Our first aim is to prove that if $\cot Y> \max\{2,r\}$ and $X$ is
infinite-dimensional, then one can find an approximable non-$r$-summing linear
operator from $X$ to $Y$. Consequences of this
result complement and generalize classical results from \cite{bennett, davis,
Lindenstrauss}. Furthermore, it is sharp in the sense that it is not valid for operators into cotype 2 spaces.

The proof of the next result is a far reaching refinement of the arguments used in
\cite{studia, ann} for spaces with unconditional Schauder basis, which, in
their turn, were inspired in the classical paper \cite{Lindenstrauss}.

\begin{theorem}\label{jun5} Let $X$ and $Y$ be infinite-dimensional Banach spaces,
$2 \leq p < \infty$ and $q,r >0$
such that $\cot Y \geq p>q \geq r.$ If ${\cal A}(X,Y) \subseteq \Pi_{q,r}(X,Y)$, then $id_{X}$ is
$(\frac{pq}{p-q},r)$-summing.
\end{theorem}

\begin{proof} Since $Y$ is infinite-dimensional, from \cite[Theorem
14.5]{DJT} we have that
\begin{equation*}
\label{cot}\cot Y = \sup\{2\leq s\leq\infty;\text{ }Y\text{ finitely
factors the formal inclusion
}\ell_{s}\hookrightarrow\ell_{\infty}\},
\end{equation*}
and from \cite[p. 304]{DJT} we know that this supremum is attained.
So $Y$ finitely factors the formal inclusion $\ell_p
\hookrightarrow\ell_{\infty}$, that is, there exist $C_{1},C_{2}>0$
such that for every $n\in\mathbb{N}$,
there are $y_{1},\ldots,y_{n} \in Y$ so that%
\begin{equation*}
C_{1}\left\Vert (a_{j})_{j=1}^{n}\right\Vert _{\infty}\leq\left\Vert
{\sum\limits_{j=1}^{n}}
a_{j}y_{j}\right\Vert \leq C_{2}\left(
{\sum\limits_{j=1}^{n}}
\left\vert a_{j}\right\vert ^{p}\right)  ^{\frac{1}{p}} \label{vvvv}%
\end{equation*}
for every $a_{1},\ldots,a_{n}\in\mathbb{K}.$

From the inequality $\| \cdot \| \leq \pi_{q,r}(\cdot)$ and the
assumption ${\cal A}(X,Y) \subseteq \Pi_{q,r}(X,Y)$ it follows easily that ${\cal A}(X,Y)$ is closed with respect to $\pi_{q,r}$. So we can use the
Open Mapping Theorem to get a constant $K>0$ such that $\pi_{q,r}(u)\leq K\Vert u\Vert$ for
all linear operators
 $u \in {\cal A}(X,Y).$ Let $n\in\mathbb{N}$
and $x_{1},\ldots,x_{n}\in X$ be given. Consider $x_{1}^{\ast},\ldots
,x_{n}^{\ast}\in B_{X^{\ast}}$ so that $x_{j}^{\ast}(x_{j})=\left\Vert
x_{j}\right\Vert $ for every $j=1,\ldots,n$. Let $\mu_{1},\ldots,\mu_{n}$ be
scalars such that $\sum\limits_{j=1}^{n}|\mu_{j}|^{s}=1,$ where $s=\frac{p}%
{q}.$ Let $$u\colon X\longrightarrow Y~,~
u(x)=\sum\limits_{j=1}^{n}\left\vert \mu_{j}\right\vert ^{\frac{1}{q}}%
x_{j}^{\ast}(x)y_{j}.$$
It is clear that $u$ is a finite rank operator, so $u\in {\cal A}(X,Y)$. Moreover, for every $x\in X$,%
\begin{align*}
\left\Vert u(x)\right\Vert  &  =\left\Vert \sum\limits_{j=1}^{n}\left\vert
\mu_{j}\right\vert ^{\frac{1}{q}}x_{j}^{\ast}(x)y_{j}\right\Vert \leq
C_{2}\left(  \sum\limits_{j=1}^{n}\left\vert \left\vert \mu_{j}\right\vert
^{\frac{1}{q}}x_{j}^{\ast}(x)\right\vert ^{p}\right)  ^{\frac{1}{p}}\\
&  \leq C_{2}\left(  \sum\limits_{j=1}^{n}\left\vert \mu_{j}\right\vert
^{\frac{p}{q}}\right)  ^{\frac{1}{p}}\left\Vert x\right\Vert%
=C_{2}\left\Vert x\right\Vert .
\end{align*}
We thus have $\pi_{q,r}(u)\leq K\Vert u\Vert\leq
KC_{2}:=K_{1}.$ Note that for $k=1,\ldots,n$, we have%
\begin{align*}
\Vert u(x_{k})\Vert &  =\left\Vert \sum\limits_{j=1}^{n}\left\vert \mu
_{j}\right\vert ^{\frac{1}{q}}x_{j}^{\ast}(x_{k})y_{j}\right\Vert \geq
C_{1}\left\Vert \left(  \left\vert \mu_{j}\right\vert ^{\frac{1}{q}}%
x_{j}^{\ast}(x_{k})\right)  _{j=1}^{n}\right\Vert _{\infty}\\
&  \geq C_{1}\left\vert \mu_{k}\right\vert ^{\frac{1}{q}}x_{k}^{\ast}%
(x_{k})=C_{1}\left\vert \mu_{k}\right\vert ^{\frac{1}{q}}\Vert x_{k}%
\Vert.
\end{align*}
So we have%
\begin{align}
\left(  \sum\limits_{j=1}^{n}\Vert x_{j}\Vert^{q}\left\vert \mu
_{j}\right\vert \right)  ^{\frac{1}{q}}  &  =\left(  \sum\limits_{j=1}%
^{n}\left(  \Vert x_{j}\Vert\left\vert \mu_{j}\right\vert ^{\frac{1}{q}%
}\right)  ^{q}\right)  ^{\frac{1}{q}}\nonumber\\
&  \leq C_{1}^{-1}\left(  \sum\limits_{j=1}^{n}\Vert u(x_{j})\Vert^{q}\right)
^{\frac{1}{q}}\label{poiu}\\
&  \leq C_{1}^{-1} \pi_{q,r}(u)\Vert(x_{j})_{j=1}%
^{n}\Vert_{w,r}.\nonumber
\end{align}
Observing that this last inequality holds whenever $\sum\limits_{j=1}^{n}%
|\mu_{j}|^{s}=1$ and that $\frac{1}{s}+\frac{1}{\frac{s}{s-1}}=1$, we obtain
\begin{align*}
\left(  \sum\limits_{j=1}^{n}  \Vert x_{j}\Vert^{\frac{s}{s-1}%
q} \right)  ^{\frac{1}{\frac{s}{s-1}}}  &  =\sup\left\{ \left\vert
\sum\limits_{j=1}^{n}\mu_{j}\Vert x_{j}\Vert^{q}\right\vert ;\sum
\limits_{j=1}^{n}|\mu_{j}|^{s}=1\right\} \\
&  \leq\sup\left\{  \sum\limits_{j=1}^{n}\left\vert \mu_{j}\right\vert \Vert
x_{j}\Vert^{q};\sum\limits_{j=1}^{n}|\mu_{j}|^{s}=1\right\} \\
&  \overset{(\ref{poiu})}{\leq}\left(  C_{1}^{-1}\pi_{q,r}(u)\|(x_{j})_{j=1}^{n}\Vert_{w,r}\right)  ^{q}\\
&  \leq\left(  C_{1}^{-1}\cdot K_{1}\cdot\Vert(x_{j})_{j=1}^{n}\Vert_{w,r}%
\right)  ^{q},
\end{align*}

and then
\[
\left(  \sum\limits_{j=1}^{n}\Vert x_{j}\Vert^{\frac{s}{s-1}q}\right)
^{\frac{1}{\frac{s}{s-1}q}}\leq C_{1}^{-1}\cdot K_{1}%
\cdot\Vert(x_{j})_{j=1}^{n}\Vert_{w,r}.
\]
Since $\frac{s}{s-1}q=\frac{pq}{p-q}$, $n$ and $x_{1},\ldots,x_{n}\in X$ are
arbitrary, we conclude that $id_{X}$ is $(\frac{pq}{p-q},r)$-summing.
\end{proof}

\begin{corollary}
\label{8jun08}Let $X$ and $Y$ be infinite-dimensional Banach spaces,
$2 \leq p < \infty$ and $q>1$ such that $\cot Y\geq p>q$. If
$\frac{pq}{p-q}>2$ and ${\cal A}(X,Y) \subseteq \Pi_{q,1}(X,Y)$, then $X$ has cotype $\frac
{pq}{p-q}.$

\end{corollary}

\begin{proof} From Theorem \ref{jun5} we conclude that $id_{X}$ is $(\frac {pq}{p-q},1)$-summing.
Since $\frac{pq}{p-q}>2$, using a result due to Talagrand (see \cite{Tal}) we conclude that $X$
has cotype $\frac{pq}{p-q}$.
\end{proof}

Observe that the result above is an interesting improvement of the linear case of \cite[Corollary
2]{studia}, because there, contrary to here, a Schauder basis for $X$ is
required. \\
\indent A well known result due to Maurey and Pisier asserts that
$\cot X = \inf \{2 \leq q \leq \infty : id_X \in \Pi_{q,1}(X,X)\}$
(\cite[Theorem 14.5]{DJT} and \cite{MPi}). Next corollary is a
significant improvement of the linear cases of \cite[Theorem
7]{studia} and \cite[Corollary 2.5]{cot-inf} and complements
information from \cite{MPi}:

\begin{corollary}
\label{gghh}Let $X$ be an infinite-dimensional Banach space. If there is an infinite-dimensional Banach space $Y$ with no finite cotype and such that ${\cal A}(X,Y) \subseteq \Pi_{q,1}(X,Y)$, then $\cot X\leq q$.
\end{corollary}

\begin{proof} The assumption $\cot Y = \infty$ allows us to make $p\longrightarrow\infty$ in Theorem \ref{jun5} to conclude that $id_{X}$ is
$(q+\varepsilon,1)$-summing for every $\varepsilon>0$. Hence
$q\geq2$ by \cite[Theorem 10.5]{DJT} and $\cot X\leq q$ by
\cite[Theorem 14.5]{DJT}.
\end{proof}
\begin{remark}\rm
Let $X$ be an infinite-dimensional Banach space with an
unconditional Schauder basis $(x_{n})$. Define
\[
\mu_{(x_{n})} = \inf\Bigl\{  t: (a_{j}) \in\ell_{t} {\rm~whenever~}\sum
_{j=1}^{\infty}a_{j}x_{j} \in X \Bigr\}  .
\]
 In \cite[Theorem 7]{studia} it is shown that if $Y$ has no finite cotype and $\Pi_{q,1}(X,Y)=\mathcal{L}(X,Y)
$, then $\mu_{(x_{n})}\leq q$.
Corollary \ref{gghh} improves this information in three directions:
(i) a Schauder basis is not needed in Corollary \ref{gghh}, (ii) if
$X$ has an unconditional Schauder basis, the conclusion $\cot X\leq
q$ of Corollary \ref{gghh} is stronger than the conclusion
$\mu_{(x_{n})}\leq q$ of \cite[Theorem 7]{studia}, (iii) it is
enough to have ${\cal A}(X,Y) \subseteq \Pi_{q,1}(X,Y)$ instead of $\Pi_{q,1}(X,Y)=\mathcal{L}(X,Y)
$.
\end{remark}

Now we are in the position to prove our main results.

\begin{theorem}
\label{kkkk} Let $X$ and $Y$ be
infinite-dimensional Banach spaces. If $\cot Y \geq p>q \geq r > \frac{2pq}{pq + 2p -2q}$, then there exists an approximable non-$(q,r)$-summing linear operator from $X$ to $Y$. 
\end{theorem}

\begin{proof} Assume that $\mathcal{A}(X,Y) \subseteq \Pi_{q,r}(X,Y)$. By Theorem \ref{jun5} we have that $id_{X}$ is
$(\frac{pq}{p-q},r)$-summing. Since $id_X \neq 0$ we have $r \leq \frac{pq}{p-q}$. By \cite[Theorem 10.5]{DJT} it follows that $\frac{1}{r} - \frac{p-q}{pq} \geq \frac{1}{2}$, that is $r \leq \frac{2pq}{pq + 2p -2q}$.
\end{proof}

A classical result due to Davis and Johnson \cite{davis} asserts
that if $X$ is superreflexive, then there exists a compact
non-$r$-summing linear operator from $X$ to any infinite-dimensional
space $Y$. Let us see that for operators with range spaces $Y$ with
$\cot Y> \max\{2,r\}$ there is no need to impose any condition on the domain space $X$:

\begin{corollary}\label{corol} Let $X$ and $Y$ be infinite-dimensional Banach spaces with $\cot Y> \max\{2,r\}$. Then there
exists an approximable (hence compact) non-$r$-summing linear operator from $X$ to $Y$.
\end{corollary}

\begin{proof} For $2\leq r<\cot Y$ the result follows from Theorem \ref{kkkk}, with $p = \cot Y$ and $q = r$, because $ \frac{2pq}{pq + 2p -2q}< 2$. The case $r<2$ is a consequence of the inclusion theorem \cite[Theorem 2.8]{DJT}.
\end{proof}

\begin{remark}\rm (a) Grothendieck's theorem $\Pi_1(\ell_1,\ell_2) = {\cal L}(\ell_1,\ell_2)$ makes clear that Theorem \ref{kkkk} and Corollary \ref{corol} are sharp in the sense that they are not valid for operators with range cotype 2
spaces. The case $\cot Y>2$ is also close to optimality. In fact,
from \cite[Corollary 10.10]{DJT} we know that if $Y$ is an
$\mathcal{L}_{q}$-space $(q>2)$ and $r>q=\cot Y$,
then%
\[
\Pi_{r}(c_{0};Y)=\mathcal{L}(c_{0};Y).
\]
(b) It is worth mentioning that ${\cal A}$ is contained in any closed operator ideal ${\cal I}$ (e.g., operators that are compact, weakly, compact, completely continuous, strictly singular, strictly cosingular, Banach-Saks property, Rosenthal property, etc - see \cite{pie}). So, for any closed operator ideal $\cal I$, if $X$ and $Y$ are infinite-dimensional Banach spaces with $\cot Y> \max\{2,r\}$, then there
exists a  non-$r$-summing linear operator from $X$ to $Y$ belonging to $\cal I$.
\end{remark}

A classical theorem due to Lindenstrauss
and Pe{\l}czy\'nski \cite[Proposition 8.1(2)]{Lindenstrauss}
asserts that if $X$ and $Y$ are infinite-dimensional and
$\Pi_1(X,Y)= \mathcal{L}(X,Y)$, then $\cot X = 2$. Part (c) of the next proposition improves this result in the sense that a stronger conclusion is obtained from a weaker assumption.

 \begin{proposition}\label{corol1}~\\
{\rm (a)} Let $X$ be an infinite-dimensional Banach space. If there is an infinite-dimensional Banach space $Y$ with finite cotype such that ${\cal A}(X,Y) \subseteq \Pi_{\frac{2\cot Y}{2+\cot Y},1}(X,Y)$, then $X$ has the Orlicz property (that is, $id_X$ is $(2,1)$-summing).\\
{\rm (b)} If $X$ is infinite-dimensional and ${\cal A}(X,Y) \subseteq \Pi_{r}(X,Y)$ for some infinite-dimensional Banach space $Y$ and some $1 \leq r < 2$, then $\cot Y = 2$.\\
{\rm (c)} If $X$ and $Y$ are infinite-dimensional Banach spaces and every approximable linear operator from $X$ to $Y$ is $1$-summing,
then $\cot X = \cot Y = 2$.
\end{proposition}

\begin{proof} (a) Choose $r = 1$, $p = \cot Y$ and $q = \frac{2p}{2+p}$ in Theorem \ref{jun5}.\\
(b) This is immediate from Theorem \ref{kkkk}. In fact, if $\cot
Y>2$, we can take $p=\cot Y$ and $q=r$ in Theorem  \ref{kkkk} and
conclude that $\mathcal{A}(X;Y)$ is not contained in $\Pi_{r}(X;Y)$.\\
(c) Putting $r = 1$ in (b) we conclude that $\cot Y = 2$, then (a)
yields that $id_X$ is $(2,1)$-summing. Therefore $\cot X = 2$ by
\cite[Theorem 14.5]{DJT}.
\end{proof}

As announced in the introduction, we shall improve substantially the following information from Bennett \cite[Proposition 5.2(i)]{bennett}:\\
 $\bullet$ $\Pi_{q,1}(\ell_1,\ell_p)\neq \mathcal{L}(\ell_1,\ell_p)$ whenever $2 \leq p < \infty$ and $q < \frac{2p}{2 + p}$;\\
 $\bullet$ $\Pi_{q,2}(\ell_1,\ell_p)\neq \mathcal{L}(\ell_1,\ell_p)$ whenever $2 \leq p < \infty$ and $q < p$.\\
 $\bullet$ $\Pi_{q,1}(\ell_1;\ell_\infty)\neq {\mathcal L}(\ell_1;\ell_\infty)$ whenever $1\leq q<2$.\\
\indent Our improvement says that $\ell_1$ may be replaced by any
infinite-dimensional Banach space, $\ell_p$ may be replaced by
any infinite-dimensional Banach space with $\cot Y = p$ and the existence of a non-absolutely summing operator can be replaced by the existence of an approximable non-absolutely summing operator:

\begin{proposition}\label{propos} Let $X$ and $Y$ be infinite-dimensional Banach spaces.  \\
{\rm (a)} If $\cot Y < \infty$, then there exists an approximable non-$(q,1)$-summing linear operator from $X$ to $Y$ for every $q < \frac{2\cot Y}{2 +\cot Y }$.\\
{\rm (b)} If $\cot Y < \infty$, then there exists an approximable non-$(q,2)$-summing linear operator from $X$ to $Y$ for every $q < \cot Y$.\\
{\rm (c)} If   $\cot Y=\infty$, then there exists an approximable non-$(q,1)$-summing linear operator from $X$ to $Y$ for every $1\leq q<2$.
\end{proposition}

\begin{proof} (a)  Make $r =1$ in Theorem  \ref{kkkk} and observe that $q < \frac{2p}{2+p} \Longleftrightarrow 1 > \frac{2pq}{pq + 2p - 2q}$.\\
(b) Just make $r = 2$ in Theorem \ref{kkkk}.\\
(c) Otherwise, from Corollary \ref{gghh} it would follow that $\cot X\leq q<2$.
\end{proof}

\begin{remark}\rm Proposition \ref{propos} is sharp: for example, from \cite[Proposition 5.2(i)]{bennett} we know that $\Pi_{\frac{2p}{2+p},1}(\ell_1,\ell_p) = \mathcal{L}(\ell_1,\ell_p)$ for $p \geq 2$ and $\Pi_{p,2}(\ell_1,\ell_p) = \mathcal{L}(\ell_1,\ell_p)$ for $p \geq 2$.
\end{remark}

A final application of our results illustrates, in one single result, the well known relevance of the space $\ell_1$ and of the concept of cotype in the theory of absolutely summing operators:

\begin{proposition} Let $X$ and $Y$ be infinite-dimensional Banach spaces. Let $2 \leq r < \cot Y$ and $q\geq r$ be such that $\Pi_{q,r}(X,Y) = \mathcal{L}(X,Y)$. Then $\mathcal{L}(\ell_1,\ell_{\cot Y}) = \Pi_{q,r}(\ell_1,\ell_{\cot Y})$.
\end{proposition}

\begin{proof} Combine Theorem \ref{kkkk} with \cite[Proposition 5.2(iv)]{bennett}.
\end{proof}

\noindent {\bf Final remark.} The results of this paper can be
adapted, {\it mutatis mutandis}, to absolutely summing homogeneous
polynomials and multilinear mappings. For homogeneous polynomials,
Theorem \ref{jun5} reads as: if $2 \leq p < \infty$ and $q>0$ are
such that $\cot Y \geq p>q \geq \frac rm$, and the the space of absolutely $(q,r)$-summing $m$-homogeneous polynomials from $X$ to $Y$ contains the closure (in the usual sup-norm) of the space of all continuous $m$-homogeneous polynomials of finite type, then $id_X$ is
$(\frac{mpq}{p-q},r)$-summing. The resulting consequences and
related results follow accordingly.\\

\vspace*{1em} \noindent[Geraldo Botelho] Faculdade de Matem\'atica,
Universidade Federal de Uberl\^andia, 38.400-902 - Uberl\^andia, Brazil,
e-mail: botelho@ufu.br.

\medskip

\noindent[Daniel Pellegrino] Departamento de Matem\'atica, Universidade
Federal da Pa-ra\'iba, 58.051-900 - Jo\~ao Pessoa, Brazil, e-mail: dmpellegrino@gmail.com.

\medskip

\noindent[Pilar Rueda] Departamento de An\'alisis Matem\'atico, Universidad de
Valencia, 46.100 Burjasot - Valencia, Spain, e-mail: pilar.rueda@uv.es.

\end{document}